\documentclass{article}
\usepackage[utf8]{inputenc}
\usepackage{xcolor}

\title{Hyperbolic branching Brownian motion: \\ the empirical limit measure}
\author{David Geldbach\footnote{Department of Statistics, University of Oxford, 24-29 St Giles, OX1 3LB Oxford, United Kingdom.}}
\date{\today}

\usepackage[utf8]{inputenc}
\usepackage[T1]{fontenc}
\usepackage[english, ngerman, english]{babel}
\usepackage[pdftex]{graphicx}
\usepackage{latexsym}
\usepackage{amsmath,amssymb,amsthm}
\usepackage{commath}
\usepackage{enumitem}
\usepackage{mathrsfs}
\usepackage[margin=1 in]{geometry}
\usepackage{dsfont}
\usepackage{xcolor}
\usepackage{hyperref}   
\usepackage{todonotes}  
\usepackage[doi=false,
            isbn=false,
            url=false,
            eprint=false, 
            maxbibnames=99,
            date=year]{biblatex}   
\usepackage{csquotes}
\usepackage{amsfonts}
\usepackage{glossaries-extra}

\hypersetup{
    colorlinks=true,
    linkcolor=black,
    filecolor=black,      
    urlcolor=black,
    citecolor=black,
    pdftitle={Notes regarding Marchals Tree Growth},
    pdfpagemode=FullScreen,
    }

\renewbibmacro{in:}{}
\AtEveryBibitem{
  \clearlist{language}
}

\graphicspath{ {./} }

\setuptodonotes{color=yellow!50, linecolor=blue!30}

\newtheorem{thm}{Theorem}
\newtheorem{lem}[thm]{Lemma}
\newtheorem{prop}[thm]{Proposition}
\newtheorem{cor}[thm]{Corollary}

\newtheorem*{conj}{Conjecture}

\theoremstyle{definition}

\newtheorem{quest}[thm]{Question}

\numberwithin{equation}{section} 
\numberwithin{thm}{section}

\newcommand*{\eps}{\ensuremath{\varepsilon}}            
\newcommand*{\CC}{\ensuremath{\mathbb{C}}} 		        
\newcommand*{\RR}{\ensuremath{\mathbb{R}}}	        	
\newcommand*{\NN}{\ensuremath{\mathbb{N}}}		        
\newcommand*{\EE}{\ensuremath{\mathbb{E}}}	        	
\newcommand*{\PP}{\ensuremath{\mathbb{P}}}	        	
\newcommand*{\DD}{\ensuremath{\mathbb{D}}}	        	
\newcommand*{\HH}{\ensuremath{\mathbb{H}}}	        	
\newcommand*{\II}{\ensuremath{\mathds{1}}}	        	
\newcommand*{\supp}{\ensuremath{\mathrm{supp} \ }}	        
\newcommand*{\diam}{\ensuremath{\mathrm{diam} \ }}	        
\newcommand*{\Var}{\ensuremath{\mathrm{Var}  }}	        

\begin{document}

\maketitle

\begin{abstract}
    We study branching Brownian motion in hyperbolic space. As hyperbolic Brownian motion is transient, the normalised empirical measure of branching Brownian motion converges to a random measure $\mu_\infty$ on the boundary. We show that the Hausdorff dimension of $ \mu_\infty$ is $(2\beta)\wedge 1$ where $\beta$ is the branching rate, and that $\mu_\infty$ admits a Lebesgue density for $\beta>1/2$. This is very different to the behaviour of the set of accumulation points on the boundary where $\beta_c=1/8$ which has been shown by Lalley and Sellke \cite{lalley_hyperbolic_1997}. This answers several questions posed by Woess \cite{woess_notes_2024} and similar questions posed by Candellero and Hutchcroft \cite{candellero_boundary_2023}. We believe that our methods also apply to branching random walks on non-elementary hyperbolic groups. 
\end{abstract}


\begin{figure}[h]
    \centering
    \includegraphics[width=0.75\linewidth]{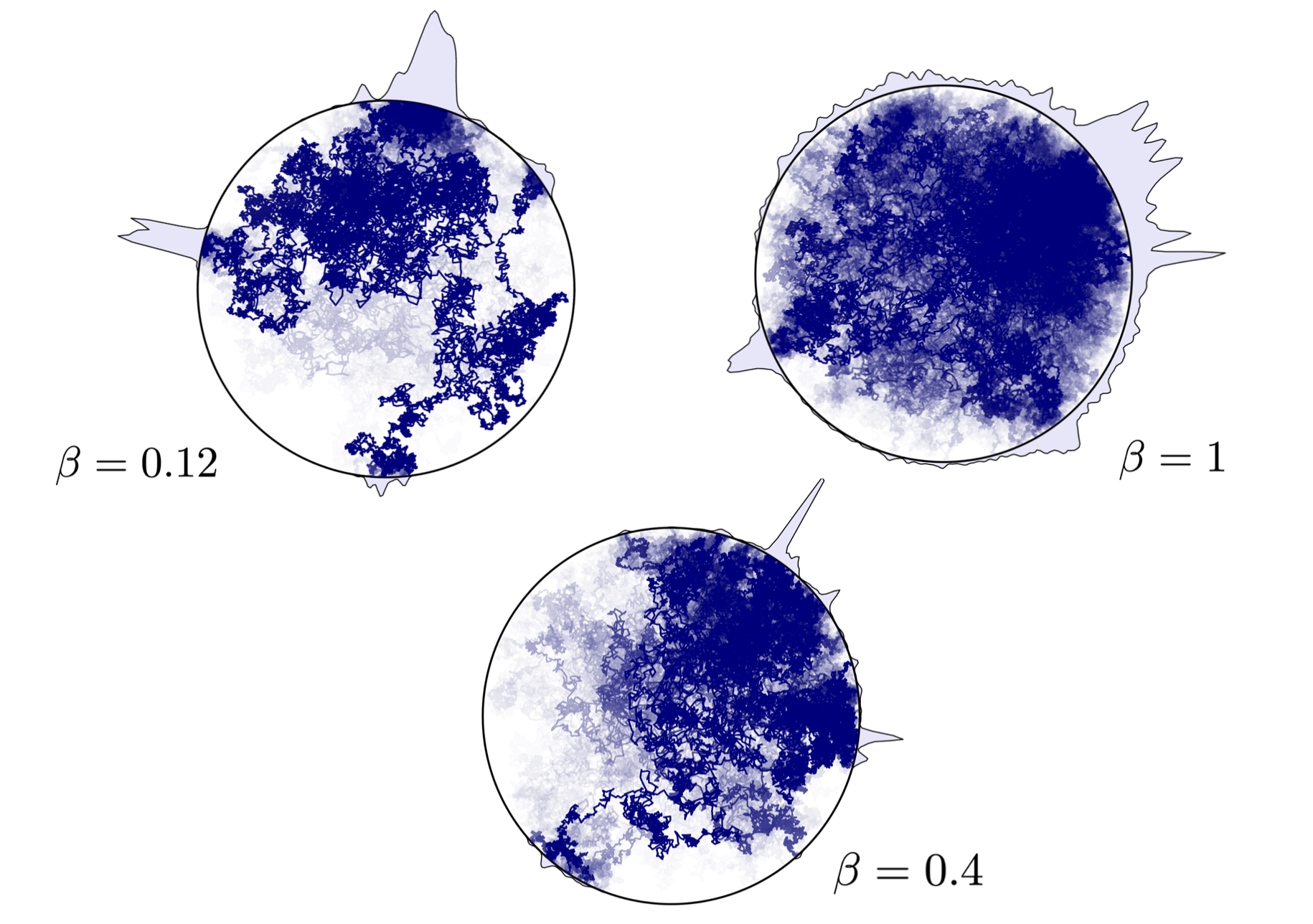}
    \caption{Three simulations of hyperbolic BBM and the limit of its empirical distribution $\mu_\infty$ on the boundary. The branching rates are $\beta \in \{ 0.12, 0.4, 1\}$ from left to right. The Hausdorff dimensions of $\mu_\infty$ are $\{0.24, 0.8, 1\}$ whereas the Hausdorff dimensions of the set of accumulation points on the boundary are $\{0.4, 1,1\}$. Observe that in the middle picture where $\beta=0.4$ there are a lot of paths accumulating on the boundary that do not contribute significantly to $\mu_\infty$.} 
    \label{fig:simulation}
\end{figure}






\section{Introduction}

Given a transient stochastic process, one can often define a natural extension of the state space so that the process converges almost surely to a point on the boundary of this space. We can think of the distribution of this limit as an exit distribution of the process. These exit distributions {form} a well studied topic in {their own right}, but it becomes even richer when combined with branching processes. The idea here is that we have multiple, possibly infinitely many, correlated particles that all escape towards the boundary. This induces different random structures on the boundary, notably the set $\Lambda$ of accumulation points on the boundary, and a random probability measure $\mu_\infty$ which is the weak limit of the empirical measures.
Loosely speaking, $\Lambda$ is determined by all particles including rare exceptional particles, while $\mu_\infty$ is determined only by the bulk of the particles. We always have that $\mu_\infty(\Lambda)=1$, but it is natural to ask if $\mu_\infty$ is actually supported on a smaller set than $\Lambda$, and if yes, how we can quantify the difference. This is the aim of this paper in the case of branching Brownian motion in hyperbolic space. 

 Branching Brownian motion (BBM) on $\RR$ is an interacting particle system. Particles move as independent Brownian motions and split in two at a given rate $\beta$. 
 Here the the maximal displacement at time $t$, is a major object of interest \cite{AidBerBruShi2013, Bra1978}. 
 On the other hand, the number of particles near the origin always grows exponentially for any $\beta>0$. 
If the underlying space is hyperbolic, the behaviour of BBM is markedly different. {Let} $\beta_c=1/8$, then for any
 $\beta<\beta_c$ the process eventually vacates any compact set almost surely. On the other hand, for any $\beta>\beta_c$, the number of particles near the origin grows exponentially {as in $\RR$}. The same is true for a discrete version of this model, a branching random walk on a homogeneous tree. We give a precise definition of branching Brownian motion in hyperbolic space in the next section, but {also} refer to the recent article by Woess \cite{woess_notes_2024} and the references therein for background on hyperbolic BBM. 

 The limit set $\Lambda$ of hyperbolic BBM {was} first studied by Lalley and Sellke \cite{lalley_hyperbolic_1997} who show{ed} that $\Lambda$ is a fractal--like random set and compute its Hausdorff dimension. (See for example \cite{falconer_fractal_2014} for some background on fractals and Hausdorff dimension.) Others have studied similar sets of accumulation points of branching random walks on the boundary on discrete hyperbolic spaces, see for example \cite{dussaule_branching_2025, hueter_anisotropic_2000,sidoravicius_limit_2023}. Much less is known about $\mu_\infty$ and the sets it's supported on. Even the existence of $\mu_\infty$ has only been shown recently \cite{candellero_boundary_2023, kaimanovich_limit_2023}. In fact,  we {believe} that this paper is the first {work} to show quantitative properties of $\mu_\infty$.

\subsection{The model}

Hyperbolic space is usually modelled with the Poincaré disk model $\DD$ or the upper half plane model $\HH$. We use them interchangeably. They are Riemannian manifolds with metrics given by 
\begin{equation*}
    \frac{2 \sqrt{dx^2+dy^2}}{1-x^2-y^2} 
    \text{ for }(x,y)\in \DD=\{(x,y)\in \RR^2: x^2+y^2<1\},
\end{equation*}
for the disk model and for the upper half plane model by
\begin{equation*}
    \frac{\sqrt{dx^2 + dy^2}}{y^2} \text{ for }(x,y)\in \HH=\{(x,y)\in \RR \times \RR_+\}.
\end{equation*}
The two models are isometric, an isometry $f:\DD \to \HH$ is given by $f(z)=i\frac{1+z}{1-z}$ where we identify $\DD$ and $\HH$ with subsets of $\CC$ by $z=x+iy$. Note that the origin $0\in \DD$ corresponds to $i \in \HH$. Both $\DD$ and $\HH$ are endowed with natural boundaries $\partial \DD$ and $\partial \HH$ given by $\partial\DD=\{z\in \CC: \vert z\vert =1\}$ and $\partial \HH=\{z\in \CC: \Im(z)=0\}$.
The hyperbolic Laplacian is given by 
\begin{equation*}
    \mathcal{L}_\DD = \frac{(1-\vert z\vert^2)^2}{4}\left( \partial_x^2 + \partial_y^2\right), \qquad \text{respectively} \qquad \mathcal{L}_\HH = y^2\left( \partial_x^2 + \partial_y^2\right).
\end{equation*}
From this we define hyperbolic Brownian motion to be the stochastic process with generator $\frac{1}{2}\mathcal{L}_\DD$ (respectively $\frac{1}{2}\mathcal{L}_\HH$). In the upper half plane model we could also do this by solving the pair of stochastic differential equations
\begin{equation*}
    dX_t = Y_t dW_t, \quad dY_t =Y_tdB_t,
\end{equation*}
where $(W_t)_t$ and $(B_t)_t$ are independent Brownian motions. This process is then canonically started from $(X_0, Y_0)=(0,1)$. Observe that $(Y_t)_t$ is a geometric Brownian motion, hence we can solve the SDE explicitly in the second coordinate,
\begin{equation*}
    Y_t = \exp\left( -\frac{t}{2} + B_t\right).
\end{equation*}
This also tells us that $X_t$ is Gaussian with mean $0$ and variance $\int_0^t \exp\left( -s + 2B_s\right) ds$, conditional on $(B_s)_{s\geq 0}$. From this we can see that hyperbolic Brownian motion converges to a random point $(X_\infty,0)$ on the boundary $\partial \HH$ where $X_\infty$ is Gaussian with (random) variance $\int_0^\infty \exp\left( -s + 2B_s\right) ds$.

Having defined hyperbolic Brownian motion, we define hyperbolic branching Brownian motion (BBM) to be the following particle process on $\DD$: At time $0$, we start with one particle at the origin. Particles move as independent hyperbolic Brownian motions. At rate $\beta$, each particle independently branches {in} two; both offspring particles branch and move independently. This results in a cloud of particles, we denote {those positions} by $((X_u(t),Y_u(t)),u\in \mathcal{N}(t) ),$ where $\mathcal{N}(t)$ is the set of particles alive at time $t$. By abuse of notation, we also denote the (isometric) process on $\HH$ by $((X_u(t),Y_u(t)),u\in \mathcal{N}(t) )$. Here the process is started from a single particle at $(0,1)$.

We can relate certain expectations for hyperbolic BBM to expectations of hyperbolic Brownian motion by the many--to--one formula,
\begin{equation}\label{eq:many-to-one}
    \EE_{(x,y)}\left[\sum_{u\in \mathcal{N}(t)      } f((X_u(s),Y_u(s))_{0\leq s\leq t})\right] = e^{\beta t}\EE_{(x,y)}\left[ f((X_s,Y_s)_{0\leq s\leq t})\right],
\end{equation}
for any $(x,y)\in \DD$ and measurable non--negative $f$. This follows from linearity due to the independence of movement and branching. 

\subsection{Results}

We are interested in the long term behaviour of $((X_u(t),Y_u(t)),u\in \mathcal{N}(t) )$, especially related to the boundary. We define the normalised empirical measure at time $t$ to be
\begin{equation*}
    \mu_t = \frac{1}{\vert \mathcal N (t)\vert}\sum_{u \in \mathcal{N}(t)}\delta_{(X_u(t),Y_u(t))}.
\end{equation*}
One can show that there is a measure $\mu_\infty$, supported on the boundary, such that $\mu_t$ converges weakly to $\mu_\infty$ with probability one, see \cite{woess_notes_2024}. 
This follows a simple argument: let $h:\DD \to\RR$ be a non-negative, bounded function which is harmonic with respect to hyperbolic Brownian motion. Then {$(e^{-\beta t}\vert\mathcal{N}(t)\vert)\int_\HH h(z)\mu_t(dz)$} is a martingale for hyperbolic BBM and hence converges almost surely. To obtain weak convergence, one then only needs to check that the space of harmonic functions is sufficiently rich. One can also see (essentially from the many--to--one formula \eqref{eq:many-to-one}) that for any measurable set $A \subseteq \DD \cup \partial \DD$ we have
\begin{equation*}
    \EE\left[ \mu_\infty(A)\right] = \PP \left(\lim_{t \to \infty} (X_t, Y_t)\in A \right),
\end{equation*}
from which it follows that $\mu_\infty$ is supported on the boundary almost surely. The goal of this paper is to better understand $\mu_\infty$. See Figure \ref{fig:simulation} for a simulation of hyperbolic BBM and $\mu_\infty$. One object that is slightly easier to understand is 
\begin{equation*}
    \Lambda = \left\{ \text{accumulation points of } ((X_u(t),Y_u(t)),u\in \mathcal{N}(t) )_{t \geq 0} \text{ in }\partial \DD\right\}.
\end{equation*}
Lalley and Sellke \cite{lalley_hyperbolic_1997} have analysed this set and shown that {its} Hausdorff dimension is almost surely given by
\begin{equation*}
    \dim \Lambda = \begin{cases}
        \frac{1}{2}(1-\sqrt{1- 8\beta}) \quad &\text{for } \quad 0<\beta \leq 1/8, \\
        1 & \text{for } \quad \beta>1/8.
    \end{cases}
\end{equation*}
Note the discontinuity at $\beta =1/8$.
Further, they have shown that for $\beta>1/8$ we actually have $\Lambda =\partial \DD$ almost surely. The threshold $1/8$ is unsurprisingly also the threshold for recurrence/transience, and at $\beta=1/8$ the process is transient. In his recent article about hyperbolic BBM \cite{woess_notes_2024}, Woess asks several questions about the relationship between $\Lambda$ and $\mu_\infty$, in particular if the dimensions of $\mu_\infty$ and $\Lambda$ agree. We answer these questions.

\begin{thm}\label{thm:main:1}
    The Hausdorff dimension of $\mu_\infty$ is almost surely given by
    \begin{equation*}
        \dim \mu_\infty = \begin{cases}
            2\beta \quad & \text{for} \quad 0 < \beta < 1/2, \\
            1 & \text{for} \quad \beta \geq 1/2. 
        \end{cases}
    \end{equation*}
    Consequently, $\dim \mu_\infty \neq \dim \Lambda$ for $\beta<1/2$.
\end{thm}

Note that this quantity is continuous in $\beta$ and that the threshold $\beta =1/8$ does not appear here. This is quite surprising given that $1/8$ is the threshold for local survival. Also note that $\lim_{\beta \to 0} \frac{\dim \Lambda}{\dim \mu_\infty} =1$. See Figure \ref{fig:dim_plot} for a plot of $\dim \mu_\infty$ compared to $\dim \Lambda$. 
We also give some more quantitative statements about the nature of $\mu_\infty$. Call
\begin{equation*}
    \theta \mapsto \mu_\infty([0,\theta]), \quad \theta \in [0,2\pi],
\end{equation*}
the (random) cumulative distribution function of $\mu_\infty$ where $[0,\theta]$ denotes the arc segment of $\partial\DD$ with angles between $0$ and $\theta$. 

\begin{thm}\label{thm:main:2}
    Almost surely the following statements hold:
    \begin{enumerate}[label=(\roman*)]
        \item For any $\beta >0$, $\mu_\infty$ is purely non-atomic.
        \item The (random) cumulative distribution function of $\mu_\infty$ is $\gamma$--Hölder--continuous for every exponent $\gamma < (1/2)\wedge(\beta/6)$.
        \item For $\beta > 1/2$, $\mu_\infty$ has a density with respect the the Lebesgue measure on $\partial \DD$.
    \end{enumerate}
\end{thm}

\begin{figure}[bth]
    \centering
    \includegraphics[width=0.6\linewidth]{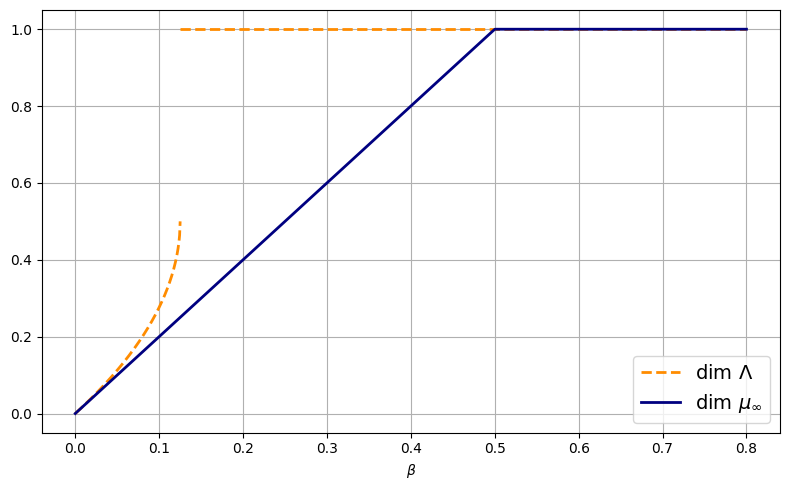}
    \caption{A plot of $\dim \mu_\infty$ and $\dim \Lambda$ as functions of $\beta$.}
    \label{fig:dim_plot}
\end{figure}

In the case $\beta =1/2$ we believe that $\mu_\infty$ should not admit a Lebesgue density almost surely though we do not prove this. 
The bound on the Hölder exponent $\gamma$ in \ref{thm:main:2} (ii) is not sharp, we believe it should hold for any $\gamma < 1 \wedge (2\beta)$.

Theorems \ref{thm:main:1} and \ref{thm:main:2} also partially answer some of the questions posed by Candellero and Hutchcroft \cite[Problems 4.2, 4.3]{candellero_boundary_2023}. In particular, they ask about the behaviour of $\mu_\infty$ for branching random walks in hyperbolic space. While they pose their questions about branching random walks in discrete time and discrete space, this should not change the overall behaviour. 

The main idea behind the proofs of Theorems \ref{thm:main:1} and \ref{thm:main:2} is that $\mu_\infty$ is determined by typical particles. In this context{,} these are particles {for which} $Y_u(t)\approx e^{-t/2}$. The structure of the paper follows this idea. In Section \ref{sec:typical_particles} we rigorously define what it means for a particle to be typical and we compute the Hausdorff dimension of the accumulation set of typical particles. In Section \ref{sec:from_typical_to_empirical} we show that indeed $\mu_\infty$ is determined by typical particles, we prove the upper bound of Theorem \ref{thm:main:1} and a sketch the lower bound. In Section \ref{sec:more_properties} we show Theorem \ref{sec:more_properties} by computing the expected moments of $\mu(I)$ for intervals $I$. As a corollary, we obtain the lower bound for Theorem \ref{thm:main:1}. Lastly we discuss some open questions in Section \ref{sec:open_questions}, in particular we discuss what should happen if you {add} a repulsive or attractive drift towards the origin, and we pose a conjecture regarding branching random walks on hyperbolic groups. 

\section{Typical particles}\label{sec:typical_particles}

We work in $\HH$.
We start by looking at \emph{typical} particles and their accumulation set on the boundary.
First, by a slight abuse of notation we let $\mathcal{N}(\infty)$ be the rays of the genealogical tree of the branching Brownian motion. This means we think of $u\in \mathcal{N}(\infty)$ as a particle with infinite trajectory $(X_u(t),Y_u(t))_{t \geq 0}$ where for fixed $t$ we let $(X_u(t),Y_u(t))$ be the position of the unique ancestor of $u$ in $\mathcal{N}(t)$.

Define the set of typical particles {to be}
\begin{equation}\label{eq:def:typical_particles}
    \mathcal{T}(K) = \left\{ u\in \mathcal{N}(\infty): \forall t\geq K: \log Y_u(t) +t/2 \in [-t^{2/3}, t^{2/3}]\right\},
\end{equation}
where $K>0$ is a parameter. We also consider the typical particles at time $t$,
\begin{equation*}
    \mathcal{T}_t(K) = \left\{u\in \mathcal{N}(t):\exists v \in \mathcal{T}(K) \ \text{such that } u \preceq v\right\}.
\end{equation*}
Note that $\mathcal{T}_t(K)$ is not measurable {with respect to} the natural filtration $(\mathcal{F}_t)_{t\geq 0}$ of the BBM. Let $u\in \mathcal{T}(K)$. Then we necessarily have that $Y_u(t) \to 0$ as $t \to \infty$ and hence $(X_u(t),Y_u(t))$ converges almost surely to a random point $(X_u(\infty), 0) \in \partial \HH$. Let
\begin{equation}\label{eq:def:typical_acc}
    \Upsilon(K)= \left\{ X_u(\infty), u\in \mathcal{T}(K) \right\}
\end{equation}
{be} the accumulation set of typical particles on the boundary. The goal of this section is to determine the Hausdorff dimension of this set. 

\begin{prop}\label{prop:dim_typical}
    For any $\beta>0$ and any $K>0$, $\dim \Upsilon(K) =2\beta \wedge 1$ almost surely on the event that $\Upsilon(K)$ in non-empty. 
\end{prop}

We show this proposition in two steps{:} Lemma \ref{lemma:upper_bound_simplified} for the upper bound and Lemma \ref{lemma:lower_bound_simplified} for the lower bound. In the following, we think of $\Upsilon(K)$ as a random subset of $\RR$ with the usual Euclidean metric.

\begin{lem}\label{lemma:diameter}
    For any $\beta<1/2$ and any $K>0$, there is $C<\infty$ such that $\EE[\diam \Upsilon (K)]\leq C$. 
\end{lem}

\begin{proof}
    Let $M_t = \sup_{u \in \mathcal{T}_t(K)} X_u(t)$ be the maximal displacement of a typical particle in the $x$--direction at time $t$. Here we use the convention $\sup \emptyset =0$ as there is nothing to show in the case when $\Upsilon(K)=\emptyset$. It suffices to show that $\EE[\limsup_{t \to \infty}M_t]<\infty$.

    Let $t\geq K$ and choose $n\in \NN_0$ such that $t\in K+[n,n+1)$. By a telescoping sum
    \begin{align}
        M_{t} &= M_K + \left(\sum_{j=0}^{n-1} (M_{K+j+1}-M_{K+j}) \right)+ (M_t-M_{K+n}). \nonumber
    \end{align}
    For each summand, we bound the difference between the maxima by the maximal positive increment at an intermediate time. This means that
    \begin{align}
         (M_{K+j+1}-M_{K+j}) &\leq \sup_{u\in \mathcal{T}_{K+j+1}(K)}\sup_{s\in[0,1]} \left(X_u(K+j+s)-X_u(K+j)\right) \nonumber\\
         &\leq \sum_{u\in \mathcal{T}_{K+j+1}(K)}\sup_{s\in[0,1]} \left(X_u(K+j+s)-X_u(K+j)\right), \nonumber
    \end{align}
    where we also used a union bound. This estimate is also true for the last summand,
    \begin{align*}
         (M_{t}-M_{K+n}) &\leq \sum_{u\in \mathcal{T}_{t}(K)}\sup_{s\in[0,t-\lfloor t\rfloor]} \left(X_u(\lfloor t\rfloor +s)-X_u(\lfloor t\rfloor)\right)\\
         &\leq \sum_{u\in \mathcal{T}_{K+n+1}(K)}\sup_{s\in[0,1]} \left(X_u(\lfloor t\rfloor+s)-X_u(\lfloor t\rfloor)\right).
    \end{align*}
    We therefore have that
    \begin{align}
        M_t &\leq M_K + \sum_{j=0}^{n} \left(\sum_{u\in \mathcal{T}_{K+j+1}(K)}\sup_{s\in[0,1]} \left(X_u(K+j+s)-X_u(K+j)\right)\right)\nonumber \\
        &\leq M_K+\sum_{j=0}^\infty \left(\sum_{u\in \mathcal{T}_{K+j+1}(K)}\sup_{s\in[0,1]} \left(X_u(K+j+s)-X_u(K+j)\right)\right), \label{epp}
    \end{align}
    note that all summands in the infinite series are non-negative and that this bound is uniform in $t\geq K$. 

    We estimate the expected value of these summands. 
    For fixed $u$ and $j$, the process \\ $\left(X_u(K+j+s)-X_u(K+j)\right)_{s\in [0,1]}$ is a Brownian motion with diffusivity $Y_u(K+j+s)$. Hence, by conditioning on $(Y_u(K+j+s))_{s\in [0,1]}$ and by the reflection principle we have the distributional equality
    \begin{align*}
        &\EE\left[\sup_{s\in[0,1]} \left(X_u(K+j+s)-X_u(K+j)\right)\middle\vert (Y_u(K+j+s))_{s\in [0,1]}\right] \\
        & \qquad \qquad \overset{d}{=} \EE\bigg[\big\vert X_u(K+j+1)-X_u(K+j)\big\vert\bigg\vert (Y_u(K+j+s))_{s\in [0,1]}\bigg].
    \end{align*}
    Now $(X_u(K+j+1)-X_u(K+j))$ is Gaussian with mean zero and with variance
    \begin{equation*}
        \int_{K+j}^{K+j+1} Y_u(s)^2ds \leq c\exp\left(-j(1+o_j(1)\right),
    \end{equation*}
    where we used that $u$ is a typical particle and where $c=c(K)$ is a constant. Therefore we get for the conditional expectation that 
    \begin{align*}
        &\EE\left[\sup_{s\in[0,1]} \left(X_u(K+j+s)-X_u(K+j)\right) \middle\vert (Y_u(K+j+s))_{s\in [0,1]}\right] \\
        &\hspace{2cm} = \EE\bigg[\big\vert X_u(K+j+1)-X_u(K+j)\big\vert \bigg\vert (Y_u(K+j+s))_{s\in [0,1]} \bigg] \\
        &\hspace{2cm} \leq \sqrt{\frac{2c}{\pi}}\exp\left(-\frac{j}{2}(1+o_j(1))\right).
    \end{align*}
    This means we also have by the tower property of conditional expectation
    \begin{align}
        &\EE\left[\sum_{u\in \mathcal{T}_{K+j+1}(K)}\sup_{s\in[0,1]} \left(X_u(K+j+s)-X_u(K+j)\right)\right] \nonumber\\
        & \hspace{0.5cm} = \EE\left[\sum_{u\in \mathcal{T}_{K+j+1}(K)} \EE \left[\sup_{s\in[0,1]} \left(X_u(K+j+s)-X_u(K+j)\right)\middle\vert (Y_u(K+j+s))_{s\in [0,1]} \right]\right] \nonumber \\
        & \hspace{0.5cm}\leq \sqrt{\frac{2c}{\pi}}\EE\left[\vert \mathcal{T}_{K+j+1}(K)\vert \right]\exp\left(-\frac{j}{2}(1+o_j(1))\right) \nonumber \\
        & \hspace{0.5cm}\leq \sqrt{\frac{2c}{\pi}}\exp\left(\beta(K+j +1)\right) \exp\left(-\frac{j}{2}(1+o_j(1))\right),
    \end{align}
    where we also used the bound 
    $$\EE\left[\vert \mathcal{T}_{K+j+1}(K)\vert \right]\leq \EE\left[\vert \mathcal{N}(K+j+1)\vert \right]=\exp\left(\beta(K+j +1)\right).$$
    We apply this to \eqref{epp} to see that uniformly in $t$
    \begin{align*}
        \EE [M_t]&\leq \EE [M_K]+\sqrt{\frac{2c}{\pi}}\sum_{j=0}^\infty \exp\left(\beta(K+j +1)\right) \exp\left(-\frac{j}{2}(1+o_j(1))\right),
    \end{align*}
    this sum converges because we assumed that $\beta<1/2$. 
    Further we have by the many--to--one formula \eqref{eq:many-to-one} that
    \begin{equation*}
        \EE [M_K] \leq \EE\left[ \sum_{u\in \mathcal{N}(K)}\vert X_u(K)\vert\right]=e^{\beta K} \EE[{\vert X_K \vert}] <\infty.
    \end{equation*}
    Looking back at \eqref{epp}, the monotone convergence theorem now implies that the right-hand side is almost surely finite. Because this is true for any $t\geq K$, this implies that $\EE[\limsup_{t \to \infty} M_t]<\infty$ by Fatou's Lemma.
\end{proof}


Clearly $\dim \Upsilon(K)\leq 1$. Therefore we need to show an upper bound for $\dim \Upsilon(K)$ only in the case $\beta<1/2$.

\begin{lem}\label{lemma:upper_bound_simplified}
    For $\beta <1/2$ and any $K>0$, $\dim \Upsilon(K) \leq 2\beta$ almost surely. 
\end{lem}

\begin{proof}
    We follow a similar idea to \cite[Proposition 11]{lalley_hyperbolic_1997}. For a particle $u\in \mathcal{T}_t(K)$, consider
    \begin{equation*}
        \Upsilon_u^t(K)=\left\{X_v(\infty):v \in \mathcal{T}(K) \text{ with } u \preceq v \right\},
    \end{equation*}
    that is{,} the {set of} limits in $\partial \HH$ of all typical descendents of $u$. Naturally, we have for any $t$ that 
    \begin{equation}\label{eq:decomposition_ups}
        \Upsilon(K) = \bigcup_{u \in \mathcal{T}_t(K)}\Upsilon_u^t(K). 
    \end{equation}
    For $u\neq v \in \mathcal{T}_t(K)$, $\Upsilon_u^t(K)$ and $\Upsilon_v^t(K)$ are independent conditional on $(X_u(t),Y_u(t))$ and $(X_v(t), Y_v(t))$.
    Let $I_u^t \subset \partial \HH$ be the smallest closed interval that contains $\Upsilon_u^t(K)$. By isometries of $\HH$, $\Upsilon_u^t(K)$ is contained in an independent copy of $\Upsilon(K)$ scaled by $Y_u(t)$, provided that {$t\geq K$}. In particular{,} we have by Lemma \ref{lemma:diameter} that for any $0<\eta \leq 1$,
    \begin{equation}\label{eq:pp}
        \EE\big[ \vert  I_u^t \vert^\eta \big\vert Y_u(t) \big] \leq Y_u(t)^\eta \EE\left[(\diam \Upsilon(K))^\eta\right] \leq (C+1) Y_u(t)^\eta.
    \end{equation}
    {Let us return} to \eqref{eq:decomposition_ups}. This decomposition implies that $\{I_u^t\}_{u \in \mathcal{T}_t(K)}$ is an interval cover for $\Upsilon(K)$. Let $\eps> 0$ such that $2\beta+\eps<1$. We apply \eqref{eq:pp} {to get}
    \begin{align*}
        \EE\left[\sum_{u\in \mathcal{T}_t(K)}\vert I_u^t\vert^{2\beta + \eps}\right] &\leq (C+1)\EE\left[\sum_{u\in \mathcal{T}_t(K)}Y_u(t)^{2\beta + \eps}\right]  \\
         &\leq (C+1) \EE\left[\vert \mathcal{T}_t(K)\vert \right] \exp\left(-(2\beta+\eps)(t/2)(1+o_t(1)\right),
    \end{align*}
    where we also used that $Y_u(t) =\exp(-(t/2)(1+o_t(1))$ for typical particles. Further we can bound $\EE\left[\vert \mathcal{T}_t(K)\vert \right] \leq \EE[\vert \mathcal{N}(t)\vert]=\exp(\beta t)$, therefore
    \begin{align}
        \EE\left[\sum_{u\in \mathcal{T}_t(K)}\vert I_u^t\vert^{2\beta + \eps}\right]  &\leq (C+1) \exp\left(\beta t - \frac{2\beta+\eps}{2}t(1+o_t(1))\right) \nonumber\\
        &= (C+1) \exp\left(- \frac{\eps}{2}t(1+o_t(1))\right).
        \label{eq:bounding_intervals}
    \end{align}
    First, this shows that 
    \begin{equation*}
        \sup_{u \in \mathcal{T}_t(K)} \vert I_u^t\vert \xrightarrow{t\to \infty} 0,
    \end{equation*}
    almost surely. 
    Next, by Fatou's Lemma and applying \eqref{eq:bounding_intervals} again, {we estimate for $2\beta+\eps$--dimensional Hausdorff measure $\mathcal{H}^{2\beta+\eps}$ of $\Upsilon(K)$},
    \begin{align*}
        \EE\left[\mathcal{H}^{2\beta+\eps}(\Upsilon(K))\right] \leq \EE\left[\liminf_{t \to \infty}\sum_{u\in \mathcal{T}_t(K)}\vert I_u^t\vert^{2\beta + \eps}\right] \leq \liminf_{t \to \infty} \EE\left[\sum_{u\in \mathcal{T}_t(K)}\vert I_u^t\vert^{2\beta + \eps}\right]=0.
    \end{align*}
    As this is true for any $\eps>0$ small enough, we obtain that $\dim \Upsilon(K) \leq 2\beta$ almost surely.
\end{proof}

\begin{lem}\label{lemma:lower_bound_simplified}
    For $\beta>0$ and any $K>0$, $\dim \Upsilon(K) \geq 2\beta \wedge 1$ almost surely on the event that $\Upsilon(K)$ is non-empty.
\end{lem}

A common tool to show lower bounds for Hausdorff dimensions is Frostman's Lemma, {a corollary of which we state below as Lemma \ref{lemma:frostman}}. See \cite[Theorem 4.13]{falconer_fractal_2014} for a reference. 

\begin{lem}\label{lemma:frostman}
    Let $A$ be a compact subset of Euclidean space. Assume that there exists a probability measure $\nu$ on $A$ such that
    \begin{equation*}
        \iint_{A\times A} \vert x-y\vert^{-\eta} \nu(dx)\nu(dy) < \infty,
    \end{equation*}
    where $\eta >0$. Then the Hausdorff dimension of $A$ is at least $\eta$.
\end{lem}

\begin{proof}[Proof of Lemma \ref{lemma:lower_bound_simplified}.]
    Throughout the proof, we work on the event that $\mathcal{T}(K)$ is non-empty. 
    To use Frostman's Lemma, we need to define a probability distribution on $\Upsilon(K)$. We do this by defining a sequence $(U_n)_{n\in \NN_0}$ of random variables such that $U_n \in \mathcal{T}_{nK}(K)$.
    \begin{enumerate}
        \item Let $U_0 = u$, where $u\in \mathcal{T}_0(K)$ is the unique initial particle. 
        \item Given $U_{n-1}$, let $U_n$ be a uniform choice from $\{u\in \mathcal{T}_{nK}(t): U_{n-1}\preceq u\}$.
    \end{enumerate}
    Let $U=\lim_{n \to \infty }U_n$ be the natural limit in $\mathcal{T}(K)$ and let $\nu$ be the distribution of $X_U(\infty)$. 

    Now let $(U'_n)_{n}$ be a copy of $(U_n)_{n}$, independent conditional on $\mathcal{T}(K)$. Let $\tau = \inf\{n: U_n\neq U_{n}'\}$, the first time that $U_n$ and $U_n'$ are different. Conditional on $\tau = n$, $X_U(\infty)-X_{U'}(\infty)$ is Gaussian with mean $0$ and variance at least
    \begin{equation*}
        \Var\left(X_U(\infty)-X_{U'}(\infty)\middle\vert \tau =n\right) \geq \int_{nK}^\infty Y_U(s)^2 +Y_{U'}(s)^2 ds \geq 2 \exp\left(-nK(1+o_n(1) )\right),
    \end{equation*}
    where we used that $Y_u(s) \geq \exp(-(s/2)(1+o_1(s)))$ for typical particles for $s\geq K$. Let $0<\eta <1$, for any Gaussian $Z$ with mean $0$ and variance $\sigma^2$ we have by Gaussian scaling that
    \begin{equation*}
        \EE[Z^{-\eta}] = \sigma^{-\eta} \EE[(\sigma^{-1}Z)^{-\eta}] = c_0\sigma^{-\eta},
    \end{equation*}
    for a universal constant $c_0=c_0(\eta)$. In particular this implies that 
    \begin{align}
        \EE\left[\vert X_U(\infty) - X_{U'}(\infty)\vert^{-\eta}\middle\vert \mathcal{T}(K),\tau =n \right]\leq c \exp(\eta nK/2),\label{eq:ce},
    \end{align}
    for a slightly different constant $c$.
    
    Next, we need to understand the distribution of $\tau$. Conditional on $\mathcal{T}$ and $(U_n)_n$ we have that
    \begin{equation*}
        \PP\left(\tau>k \middle\vert \mathcal{T}, (U_n)_n\right) =\PP\left(\forall i\leq k: U_j'=U_j \middle\vert \mathcal{T}, (U_n)_n\right)= \prod_{j=1}^n \frac{1}{N_j},
    \end{equation*}
    where $N_j = \# \left\{ u\in \mathcal{T}_{jK}:U_{j-1} \preceq u\right\}$, the number of descendants of $U_{j-1}$ in $\mathcal{T}_{jk}(K)$. {By Lemma \ref{lemma:typical_growth} below}
    \begin{equation*}
        \lim_{n \to \infty}\left(\prod_{j=1}^n \frac{1}{N_j}\right)^{1/n} = \exp(-\beta K),
    \end{equation*}
    almost surely on the event that $\mathcal{T}$ is non-empty.  Combining this with \eqref{eq:ce} yields
    \begin{align*}
        &\EE\left[\vert X_U(\infty) - X_{U'}(\infty)\vert^{-\eta}\middle\vert \mathcal{T}(K) \right] \\
        & \hspace{3cm}\leq \sum_{n=1}^\infty \EE\left[\vert X_U(\infty) - X_{U'}(\infty)\vert^{-\eta}\middle\vert \mathcal{T}(K) , \tau=n\right]\PP\left(\tau\geq n \middle\vert \mathcal{T}(K)\right)\\
        & \hspace{3cm} \leq c\sum_{n=1}^\infty \exp(\eta nK/2)\exp(-\beta K n(1+o_n(1))\\
        & \hspace{3cm} < \infty,
    \end{align*}
    for some $c>0$ and where in the last step we used that $\eta <2 \beta$. By Frostman's Lemma, Lemma \ref{lemma:frostman}, this now implies that $\Upsilon(K)$ has dimension at least $\eta$ for any $\eta$ that satisfies $\eta <2\beta$ and $\eta<1$.
\end{proof}

\begin{lem}\label{lemma:typical_growth}
    In the setting on the previous proof, $$\lim_{n \to \infty}\left(\prod_{j=1}^n \frac{1}{N_j}\right)^{1/n} = \exp(-\beta K),$$ almost surely on the event that $\mathcal{T}$ is non-empty. 
\end{lem}

\begin{proof}
    We sketch a proof of this fact, this proof can be made rigorous by carefully applying the strong law of large numbers. We provide only a sketch because Lemma \ref{lemma:lower_bound_simplified} will not be used to show the lower bound of Theorem \ref{thm:main:1}. We comment on this at the end of Section \ref{sec:from_typical_to_empirical}.
    
    The key idea is that we can think of $\mathcal{T}(K)$ as a branching Brownian motion with space and time dependent branching rate. Let
    \begin{equation*}
        \mathcal{W}(K) = \left\{(x,y,t) \in \HH \times [0,\infty): \forall t\geq K: \log(y)+t/2\in [-t^{2/3},t^{2/3}]\right\},
    \end{equation*}
    the space-time envelope of the definition of $\mathcal{T}(K)$. We also let 
    \begin{equation*}
        \phi(x,y,t)=\PP_{(x,y,t)}\left(\forall s\geq t: (X_s,Y_s, s)\in \mathcal{W}(K)\right),
    \end{equation*}
    the probability that a hyperbolic Brownian motion started from $(x,y)$ at time $t$ stays in $\mathcal{W}(K)$ forever. Importantly, $\phi(0,1,0)> 0$, that is with positive probability the initial particle stays in $\mathcal{W}(K)$. We now describe a new BBM:
    \begin{enumerate}
        \item At time $0$, we start with one particle at $(0,1)$. 
        \item All particles move as independent hyperbolic Brownian motions conditioned to stay in $\mathcal{W}(K)$.
        \item Particles branch into two at rate $\beta \phi(x,y,t)$.
    \end{enumerate}
    One can show that this modified BBM has the {same} law as the homogenous hyperbolic BBM restricted to $\mathcal{T}(K)$. Now look at a marked particle in the modified BBM, that is the initial particle is marked and{,} when it splits{,} the mark follows one of the offspring particles chosen uniformly. This is similar to the construction of $(U_n)_n$ in the previous proof. Let $(X_t^*,Y_t^*)_{t\geq 0}$ be the path of the marked particle. In fact, we have that 
    \begin{equation*}
        \lim_{n \to \infty}\left(\prod_{j=1}^n \frac{1}{N_j}\right)^{1/n} = \lim_{t \to \infty} \exp\left(-\frac{1}{t}\int_0^{Kt} \beta \phi(X_s^*,Y_s^*,s)ds\right) = \exp\left(-K\beta\right).
    \end{equation*}
    The reason for this is that along the marked path, we have that $\phi(X_s^*,Y_s^*,s)\to 1$ almost surely. This is because $\log(Y_s^*)+s/2$ will be of order $ s^{1/2}\ll s^{2/3}$ so it is very likely that for large $s$ a particle started from $(X_s^*,Y_s^*,s)$ will stay in $\mathcal{W}$ forever. 
\end{proof}

\section{From typical particles to empirical measure} \label{sec:from_typical_to_empirical}

In the previous section we analysed the accumulation set of typical particles {whose} definition we recall from \eqref{eq:def:typical_particles}. We slightly modify this. {F}or any $K>0$ and $t\geq K$, let
\begin{equation*}
        \mathcal{T}_t^\leq (K)= \left\{ u\in \mathcal{N}(t): \forall s\in [K,t]: \log Y_u(s)+s/2 \in [-s^{2/3}, s^{2/3}]\right\}.
\end{equation*}
The advantage of this modification is that $\mathcal{T}_t^\leq (K)$ is $\mathcal{F}_t$--measurable where $(\mathcal{F}_t)_{t\geq 0}$ is the natural filtration of the BBM. Similar to $\mu_t$, we define the empirical measure of typical particles at time $t$,
\begin{equation*}
    \mu_t^K = \frac{1}{\vert \mathcal N (t)\vert}\sum_{u \in \mathcal{T}_t^\leq (K)}\delta_{(X_u(t),Y_u(t))}.
\end{equation*}
Note that we chose to normalise this by $\vert \mathcal N (t)\vert$ which means that $\mu_t^K$ is a sub-probability measure. The following proposition states that $\mu_\infty$ is determined by typical particles. In some sense this is a refinement of \cite[Theorem 6.13]{woess_notes_2024} which states that a typical sample of $\mu_t$ moves at velocity $1/2$ in the hyperbolic metric.

\begin{prop}\label{prop:conv_typical_to_mu}
    Almost surely, there exists a family of sub-probability measures $(\mu^K_\infty)_{K>0}$ such that for every $K>0$
    \begin{equation*}
        \mu_t^K \rightarrow \mu_\infty^K,
    \end{equation*}
    weakly, as $t\to \infty$. Furthermore, for $K<K'$ we have $\mu^K \leq \mu^{K'}$ and as $K \to \infty$,
    \begin{equation*}
        \mu_\infty^K \rightarrow \mu_\infty. 
    \end{equation*}
\end{prop}

\begin{proof}
    Let $h:\HH \to \RR$ be a non-negative, bounded, $\mathcal{C}^2$ function which is harmonic for hyperbolic Brownian motion. That is, for all $(x,y)\in \HH$, we have $\EE_{(x,y)}[h(X_t,Y_t)]=h(x,y)$. 

    It suffices to check weak convergence only on harmonic functions as these are dense in $\mathcal{C}_b(\overline\DD,\RR)$ with respect to the uniform topology. This is because if $h$ is harmonic for the hyperbolic Laplacian, it is also harmonic for the Euclidean Laplacian on $\DD$. The harmonic functions for the Euclidean Laplacian are well understood, for example any $h$ can be written as $\Re(f)$ where $f$ is a holomorphic function.
    This fact is also discussed in the proof of \cite[Theorem 6.7]{woess_notes_2024}.
    
    We define
    \begin{equation*}
        M^K_h(t) =  \frac{\vert\mathcal{N}(t) \vert}{e^{\beta t}}\left\langle h, \mu_t^K\right\rangle = \frac{1}{e^{\beta t}}\sum_{u \in \mathcal{T}_t^\leq (K)}h(X_u(t),Y_u(t)). 
    \end{equation*}
    Then $(M_h^K(t))_{t\geq 0}$ is a non-negative supermartingale with respect to the natural filtration of hyperbolic BBM. Indeed,
    \begin{align*}
        \EE\left[M_h^K(t)\middle\vert \mathcal{F}_s\right] &= \frac{1}{e^{\beta s}}
        \hspace{-3pt}
        \sum_{u \in \mathcal{T}_s^\leq (K)}\EE_{(X_u(s),Y_u(s))}\bigg[\frac{1}{e^{\beta(t-s)}}\sum_{\substack{u \in \mathcal{T}_t^\leq (K) \\ u\preceq v}} h(X_v(t), Y_v(t))\bigg\vert \mathcal{F}_s\bigg] \\
        &=\frac{1}{e^{\beta s}}
        \hspace{-3pt}
        \sum_{u \in \mathcal{T}_s^\leq (K)}\EE_{(X_u(s),Y_u(s))}\left[h(X_t,Y_t) \II_{\{\forall r \in [s\vee K,t]:\log( Y_r) +r/2 \in [-r^{2/3},r^{2/3}]\}} \right],
    \end{align*}
    where we used the Markov property and the many--to--one {formula} \eqref{eq:many-to-one}. Next we use that $h\geq 0$ and that $h$ is harmonic,

    \begin{align*}
        \EE\left[M_h^K(t)\middle\vert \mathcal{F}_s\right]&\leq \frac{1}{e^{\beta s}}\sum_{u \in \mathcal{T}_s^\leq (K)}\EE_{(X_u(s),Y_u(s))}\left[h(X_t,Y_t)  \right] = M_h^K(s).
    \end{align*}

     Because $M^K_h(t)$ is a non-negative, uniformly integrable supermartingale (it is bounded by $\|{h}\|_{\infty} e^{-\beta t}\vert\mathcal{N}(t)\vert$ which is uniformly integrable), it converges almost surely and in $L^1$ to a limit, call it $M^K_h(\infty)$. 
     Furthermore, let $W= \lim_{t\to\infty}e^{-\beta t}\vert \mathcal{N}(t)\vert$, where the limit is almost sure and $0<W<\infty$ almost surely. Combining {these limits} gives us the almost sure limit as $t\to\infty$,
     \begin{equation*}
         \lim_{t\to\infty}\left\langle h, \mu_t^K\right\rangle  = W^{-1}M_h^K(\infty). 
     \end{equation*}
     Because $h$ is arbitrary, this implies that there is $\mu^K_\infty$ such that {almost surely} $\mu^K_t$ converges weakly to $\mu^K_\infty$. Next, for $K < K'$ and any $h$ we have that $M^K_h(t) \leq M^{K'}_h(t)$ and consequently $M^K_h(\infty) \leq M_h^{K'}(\infty)$ almost surely. This implies that $\mu_\infty^K \leq \mu_\infty^{K'}$.

     Lastly, to show that $\mu_\infty^K \to \mu_\infty$, it suffices to show that 
     \begin{equation*}
         \langle1,\mu_\infty^K\rangle \to 1,
     \end{equation*}
     almost surely as $t \to \infty$. Because $\II(x,y)=1$ for all $(x,y)\in \HH$ is harmonic, this is equivalent to showing that $M^K_\II(\infty) \to W$ as $K \to \infty$. We know that $M^K_\II(\infty) \leq W$, therefore it is enough to show that $\lim_{K \to \infty} \EE[M^K_\II(\infty)] = \EE[W]=1$. By $L_1$--convergence and the many--to--one lemma \eqref{eq:many-to-one} we have
     \begin{align*}
         \lim_{K \to \infty} \EE\left[M^K_\II(\infty)\right] & = \lim_{K \to \infty} \lim_{t \to \infty} \EE\left[M^K_\II(t)\right] \\&
         = \lim_{K \to \infty} \lim_{t \to \infty} \PP\left(\forall s \in [K,t]: \log (Y_s)+s/2 \in [-s^{2/3},s^{2/3}]\right) \\
         &= \lim_{K \to \infty}  \PP\left(\forall s \geq K: \log (Y_s)+s/2 \in [-s^{2/3},s^{2/3}]\right) \\
         &=1,
     \end{align*}
     where we recall that $(\log (Y_s)+s/2 )_{s\geq 0}$ is a standard Brownian motion and hence the {complement of the} last probability decays like $\exp(-cK^{1/3})$.     
\end{proof}

We can combine this with Proposition \ref{prop:dim_typical} to obtain the upper bound in Theorem \ref{thm:main:1}.

\begin{cor}
    For any $\beta<1/2$, we almost surely have that $\dim \mu_\infty\leq 2\beta$.
\end{cor}

\begin{proof}
    Recall the definition of $\Upsilon(K)$ from \eqref{eq:def:typical_acc}, the set of accumulation points on $\partial \HH$ of particles counted in $\mu^K_t$. By definition, we have $\mu_\infty^K(\RR \backslash \Upsilon(K)) = 0$. By the monotonicity in Proposition \ref{prop:conv_typical_to_mu} we get that
    \begin{equation*}
        \mu_\infty\left( \bigcup_{K=1}^\infty \Upsilon(K)\right)=1.
    \end{equation*}
    Therefore,
    \begin{equation*}
        \dim\mu_\infty \leq \dim \bigcup_{K=1}^\infty \Upsilon(K) = \sup_{K\in \NN} \dim \Upsilon(K) = 2\beta,
    \end{equation*}
    where we used that the Hausdorff dimension of a countable union is the supremum of {the} Hausdorff dimensions, and that $\dim \Upsilon(K) = 2\beta$ almost surely on the event that $\Upsilon(K)$ is non-empty by Proposition \ref{prop:dim_typical}. {Moreover,} we almost surely have that for $K$ large enough $\Upsilon(K)$ is non-empty.
\end{proof}

We could do a similar proof for the lower bound in Theorem \ref{thm:main:1}. Here we have the bound
\begin{equation*}
    \dim \mu_\infty \geq \dim  \mu_\infty^K. 
\end{equation*}
The issue is that it should hold that $\dim  \mu_\infty^K = \dim \Upsilon(K)$, {although} this is not obvious. Nevertheless, this is true: in the proof of Lemma \ref{lemma:lower_bound_simplified} we construct a probability measure $\nu$ on $\Upsilon(K)$ to then apply Frostman's Lemma for a lower bound on the Hausdorff dimension of $\Upsilon(K)$. A similar proof yields a lower bound on $\dim \nu$. One can see that $\nu$ is actually absolutely continuous with respect to $\mu^K_\infty$, hence the lower bound on the Hausdorff dimension also applies to $\mu^K_\infty$. We leave the details of this to the reader and instead provide a proof of the lower bound using different methods in the next section.

\section{More properties of \texorpdfstring{$\mu_\infty$}{mu}}\label{sec:more_properties}

In this section we show the lower bound of Theorem \ref{thm:main:1} as well as the other properties of $\mu_\infty$ which we claimed in Theorem \ref{thm:main:2}. The key tool is a second moment computation. 

\begin{lem}[many-to-two] \label{lemma:many_to_two}
    {Fix $r\geq 0$,} under $\PP^r$, let $(X_s^1,Y_s^1)_{s\geq 0}$ and $(X_s^2,Y_s^2)_{s\geq 0}$ be two hyperbolic Brownian motions that move together until time $r$ and afterwards move independently. Then we have for any interval $I\subseteq \RR$ that 
    \begin{equation*}
        \EE\left[\mu_\infty(I)^2\right] = 2\beta\int_0^\infty\PP^r\left(X^1_\infty,X_\infty^2\in I\right) e^{-\beta r}dr.
    \end{equation*}
    Similarly for any $K>0$,
        \begin{align*}
        &\EE\left[\mu_\infty^K(I)^2\right] = 2\beta\int_0^\infty\PP^r\bigg(X^1_\infty,X_\infty^2\in I, \\ 
        & \hspace{4cm}\forall s\geq K: \log(Y_s^1)+s/2, \log(Y_s^2)+s/2 \in [-s^{2/3},s^{2/3}]\bigg) e^{-\beta r}dr.
    \end{align*}
\end{lem}

\begin{proof}
    This is a variant of the classical many--to--two lemma. See \cite{harris_many--few_2017} for the statement and proof in the general setting. Applying the many--to--two lemma to $\mu_t$ yields
    \begin{align*}
        \EE\left[\mu_t(I\times \RR)^2\right] = 2\beta\int_0^t\PP^r\left(X^1_t,X_t^2\in I\right) e^{-\beta r}dr + e^{-\beta t} \PP\left(X_t^1 \in I\right).
    \end{align*}
    The dominated convergence theorem now completes the proof. We proceed analogously for $\mu^K_\infty.$
\end{proof}

To compute $\EE\left[\mu_\infty(I)^2\right]$, we need an estimate on hyperbolic Brownian motion. This is related to \cite[Lemma 6]{lalley_hyperbolic_1997} where {the authors} used geometric arguments to show that $X_\infty$ has a bounded density with respect to the Lebesgue measure. We determine some dependence on the starting position. 

\begin{lem}\label{lem:bound_on_bm}
    For any $x\in \RR$, $y\in (0,\infty)$ and any interval $I\subseteq \RR$ we have that
    \begin{equation*}
        \PP_{(x,y)} (X_\infty\in I)\leq\left( \frac{1}{\pi} \frac{\vert I \vert}{y} \right)\wedge 1,
    \end{equation*}
    {where $\vert I \vert$ is the length of the interval.}
\end{lem}

\begin{proof}
    First, by the isometries of $\HH$ we have that
    \begin{equation}\label{eq:tefe}
        \PP_{(x,y)} (X_\infty\in I) = \PP_{(0,1)} (X_\infty\in y^{-1}(I-x)),
    \end{equation}
    where $y^{-1}(I-x)=\{a\in \RR:ya+x\in I\}$. Next, it is known that $X_\infty$ follows a standard Cauchy distribution under $\PP_{(0,1)}$, this means that for any $a<b$ we have
    \begin{equation*}
        \PP_{(0,1)}(X_\infty \in [a,b]) =  \frac{1}{\pi}\left(\arctan b - \arctan a\right) \leq \frac{b-a}{\pi}.
    \end{equation*}
    Applying this to \eqref{eq:tefe} and adding the trivial bound $\PP(\ldots)\leq 1$ yields the lemma.
\end{proof}

\begin{prop}\label{prop:2nd_moment} $\hspace{1cm} $
    \begin{enumerate}
        \item [(i)]
    For any $\beta>0$ and $\beta \neq 3$ there is $C_1<\infty$ such that
    \begin{equation*}
        \limsup_{\eps \to 0} \sup_{\substack{I\subseteq \RR \\ \vert I \vert =\eps}} \frac{\EE\left[ \mu_\infty(I)^2\right]}{\eps^{2 \wedge(1+\beta/3)}} \leq C_1<\infty.
    \end{equation*}
    For $\beta=3$, replace $\eps^2$ above by $\eps^2 \log(\eps^{-1})$.

    \item [(ii)]
    For any $\beta>0 $, any $K>0$, and any $\delta>0$, there is $C_2=C_2(\beta, K, \delta)<\infty$ such that 
    \begin{equation*}
        \limsup_{\eps \to 0} \sup_{\substack{I\subseteq \RR \\ \vert I \vert =\eps}} \frac{\EE\left[ \mu^K_\infty(I)^2\right]}{\eps^{2 \wedge(1+2\beta-\delta)}} \leq C_2<\infty.
    \end{equation*}

    \end{enumerate}
\end{prop}

Note that there is a discrepancy between $\mu^K_\infty$ and $\mu_\infty$: the exponent for $\mu^K_\infty$ is $2\wedge(1+2\beta-\delta)$ for any arbitrary $\delta>0$ whereas for $\mu_\infty$ it is $2\wedge(1+\beta/3)$. We believe that the exponent $2\wedge(1+2\beta-\delta)$ should also apply to $\mu_\infty$ but this would require better estimates, for example a uniform control in $C_2(K)$\footnote{As pointed out by a reviewer, an alternative approach would be to consider the Fourier transform of $\mu_\infty$ (which is surprisingly tractable) to try to improve the exponent.}. For $\mu_\infty$ and $\beta=3$, the extra logarithmic factor is an artifact of{} suboptimal estimates in the proof. 

\begin{proof}[Proof of Proposition \ref{prop:2nd_moment} (i).] Assume for now that $\beta\neq 3$.
    Without loss of generality, we consider $I=[-\eps/2,\eps/2]$. 
    
    We use Lemma \ref{lemma:many_to_two} and subdivide the integral into three parts, $J_1=[0,1],J_2=[1,\log(\eps^{-2})]$ and $J_3=[\log(\eps^{-2}), \infty)$. For $i\in \{1,2,3\}$, let
\begin{equation*}
    T_i = \int_{J_i} \PP^r\left(X^1_\infty,X_\infty^2\in I\right) e^{-\beta r}dr.
\end{equation*}
We start with $T_1$. Here we bound $e^{-\beta r} \leq 1$ and then condition on the splitting position at time $r$,
\begin{align*}
    T_1 = \int_0^1 \int_{\HH}  \PP_{(x,y)}\left(X_\infty \in I\right)^2 \PP\left((X_r,Y_r) \in (dx,dy)\right) dr.
\end{align*}
We apply Lemma \ref{lem:bound_on_bm} {to obtain}
\begin{equation}\label{eq:T1}
    T_1 \leq c_1 \eps^2\int_0^1 \EE\left[Y_r^{-2}\right]dr \leq c_2 \eps^2,
\end{equation}
for some {constants} $c_1,c_2>0$. 

Next we consider $T_2$. By Lemma \ref{lem:bound_on_bm},
\begin{align*}
    \PP^r\left(X^1_\infty,X_\infty^2\in I\right) &=\int_\HH \PP((X_r,Y_r)\in(dx,dy))\PP_{(x,y)}(X_\infty \in I)^2 \\
    &\leq \int_\HH \PP((X_r,Y_r)\in(dx,dy))\PP_{(x,y)}(X_\infty \in I)\PP_{(r,0,y)}(X_\infty \in I)\\
    &\leq \int_\HH \PP((X_r,Y_r)\in(dx,dy))\PP_{(x,y)}(X_\infty \in I)\left(\frac{1}{\pi}\frac{\eps}{y} \wedge 1\right).
\end{align*}
{Hence}
\begin{align*}
    T_2 \leq \int_1^{\log(\eps^{-2})}\EE\left[\II_{X_\infty \in I} \left(\frac{1}{\pi}\frac{\eps}{Y_r} \wedge 1\right)\right] e^{-\beta r}dr.
\end{align*}
We split this integral into two parts: for $r\in [1,\log(\eps^{-1/3})]$ we bound
\begin{align*}
    \EE\left[ \II_{X_\infty\in I}\left(\frac{1}{\pi}\frac{\eps}{Y_r} \wedge 1\right)\right] &\leq \frac{1}{\pi}\eps\EE\left[\II_{X_\infty\in I}  Y_r^{-1}\right] = \frac{1}{\pi}\eps \EE\left[\PP(X_\infty\in I \vert Y_r) 
    Y_r^{-1}\right]\\
    &\leq \frac{1}{\pi^2}\eps^2 \EE\left[
    Y_r^{-2}\right] = \frac{1}{\pi^2}\eps^2 e^{3r},
\end{align*}
where we used Lemma \ref{lem:bound_on_bm} again, and where we recall that $Y_r^{-1} = \exp(r/2-B_r)$.
{F}or $r\in  [\log(\eps^{-1/3}),\log(\eps^{-2})]$, {we estimate}
\begin{equation*}
    \EE\left[ \II_{X_\infty\in I}\left(\frac{1}{\pi}\frac{\eps}{Y_r} \wedge 1\right)\right] \leq \PP(X_\infty \in I) \leq \frac{1}{\pi}\eps,
\end{equation*}
where we also used Lemma \ref{lem:bound_on_bm}.
Then our estimate on $T_2$ becomes
\begin{align}
    T_2 &\leq \frac{1}{\pi^2}\eps^2\int_1^{\log(\eps^{-1/3})} e^{3r}e^{-\beta r}dr + \frac{1}{\pi}\eps\int_{\log(\eps^{-1/3})}^{\log(\eps^{-2})}e^{-\beta r}dr \nonumber\\
    &\leq c_3\left(\eps^2 + \eps^{1+\beta/3} + \eps^{1+2\beta}\right), \label{eq:T2}
\end{align}
for some $c_3>0$.
For $T_3$, we estimate using Lemma \ref{lem:bound_on_bm}
\begin{equation*}
    \PP^r(X_\infty^1, X_\infty^2 \in I)\leq \PP(X_\infty^1\in I) \leq \frac{\eps}{\pi}.
\end{equation*}
Therefore
\begin{align}\label{eq:T3}
    T_3\leq \frac{1}{\pi}\eps \int_{\log(\eps^{-2})}^\infty e^{-\beta r}dr = c_4 \eps^{1+2\beta},
\end{align}
for $c_4>0$. 
To complete the proof, we combine \eqref{eq:T1}, \eqref{eq:T2} and \eqref{eq:T3}.

Lastly, in the case where $\beta=3$, the final estimate on $T_2$ is 
\begin{align}
    T_2 &\leq \frac{1}{\pi^2}\eps^2\int_1^{\log(\eps^{-1/3})} e^{3r}e^{-3 r}dr + \frac{1}{\pi}\eps\int_{\log(\eps^{-1/3})}^{\log(\eps^{-2})}e^{-3 r}dr \nonumber\leq c_3 \eps^2 \log(\eps^{-1}). 
\end{align}
\end{proof}

\begin{proof}[Proof of Proposition \ref{prop:2nd_moment} (ii).]
    Without loss of generality, consider $I=[-\eps/2,\eps/2]$. 
    We proceed as in the proof for (i), splitting the integral of Lemma \ref{lemma:many_to_two} into $T_1,T_2,T_3$. For ease of notation, assume that $K=1$, otherwise set $T_1=[0,K]$ and $T_2=[K,\log(\eps^{-1})]$. We use the same bounds on $T_1$ and $T_3$ but a different one on $T_2$. For $r\in [1,\log(\eps^{-2})]$, we integrate over the splitting location

    \begin{align}
        &\PP^r\left(X^1_\infty,X_\infty^2\in I,\forall s\geq 1: \log(Y_s^1)+s/2, \log(Y_s^2)+s/2 \in [-s^{2/3},s^{2/3}]\right) \nonumber\\ 
        &\hspace{3cm}\leq \int_{\RR\times [e^{-s/2-s^{2/3}},e^{-s/2+s^{2/3}}]} \PP((X_r,Y_r)\in(dx,dy))\PP_{(x,y)}(X_\infty \in I)^2 , \label{eq:te}
    \end{align}
    where the inequality comes from the fact that {we} kept the path restriction for $Y$ only for the splitting location. By Lemma \ref{lem:bound_on_bm} we have for any $y\in [e^{-s/2-s^{2/3}},e^{-s/2+s^{2/3}}]$ that
    \begin{equation*}
        \PP_{(x,y)}(X_\infty \in I) \leq \frac{1}{\pi}\eps e^{r/2+r^{2/3}}.
    \end{equation*}
    We apply this only to one factor of $\PP_{(r,x,y)}(X_\infty \in I)$ in \eqref{eq:te},
    \begin{align*}
        &\PP^r\left(X^1_\infty,X_\infty^2\in I,\forall s\geq 1: \log(Y_s^1)+s/2, \log(Y_s^2)+s/2 \in [-s^{2/3},s^{2/3}]\right) \nonumber\\ 
        &\hspace{2cm}\leq \frac{1}{\pi}\eps e^{r/2+r^{2/3}}\int_{\RR\times [e^{-s/2-s^{2/3}},e^{-s/2+s^{2/3}}]} \PP((X_r,Y_r)\in(dx,dy))\PP_{(x,y)}(X_\infty \in I) \\
        &\hspace{2cm}\leq \frac{1}{\pi}\eps e^{r/2+r^{2/3}}\PP(X_\infty \in I) \\
        &\hspace{2cm}\leq \frac{1}{\pi^2}\eps^2 e^{r/2+r^{2/3}},
    \end{align*}
    where we applied Lemma \ref{lem:bound_on_bm} again, this time with $(x,y)=(0,1)$. This then yields the following bound on $T_2$,
    \begin{align*}
        T_2 &\leq  \frac{1}{\pi^2}\eps^2\int_1^{\log(\eps^{-2})}e^{r/2+r^{2/3}} e^{-\beta r} dr \leq C(\delta)\left(\eps^2 + \eps^{1+2\beta-\delta} \right),
    \end{align*}
    for any $\delta>0$ and a constant $C(\delta)$. Here we used that $\exp(\log(\eps^{-1})^{2/3})$ grows slower than $\eps^{-\delta}$ for any $\delta$ as $\eps\to0$. Combining this with \eqref{eq:T1} and \eqref{eq:T3} completes the proof. 
\end{proof}

We can use the same methods to derive bounds on the expected $k$--th moment of $\mu_\infty^K(I)$. 

\begin{lem}\label{lemma:k_th_moment}
    For any $\beta>0$, any $K>0$, any $k\in \NN_{\geq 2}$, and any $\delta>0$, there is $C_3 =C_3(\beta, K, k, \delta)< \infty$ such that
    \begin{equation*}
        \limsup_{\eps \to 0} \sup_{\substack{I\subseteq \RR \\ \vert I \vert =\eps}} \frac{\EE\left[ \mu^K_\infty(I)^k\right]}{\eps^{k \wedge(1+2\beta (k-1)-\delta)}} \leq C_3<\infty.
    \end{equation*}
\end{lem}

\begin{proof}
    Due to considering the $k$--th moment we now need the many--to--few lemma. This is tedious to state, so
    {we only present the consequences that we need, the precise formulation can be found in \cite{harris_many--few_2017}.}
    To state the many--to--few lemma, we need to describe the joint law of $k$ hyperbolic Brownian motions. We do this by describing the behaviour of $k$ marks $1,\ldots,k$.
    \begin{enumerate}
        \item We start with one particle carrying all marks. 
        \item All particles move as independent hyperbolic Brownian motions, branching at rate $\beta$. 
        \item For a particle carrying $j$ marks, at a branching event, each mark is independently attached to one of the two offspring particles with equal probability. 
    \end{enumerate}
    Let $(X^i_t,Y^i_t)_{t\geq 0}^{1\leq i \leq k}$ denote the positions of the marks. The many--to--few lemma states that there is an explicit function $g((X^i_t,Y^i_t;1\leq i\leq k)$ such that
    \begin{align}\label{eq:crude_many_to_k}
        &\EE\left[\mu^K(I)^k\right]=\EE\bigg[\II_{ \bigcap_{i=1}^k \{X_\infty^i\in I\}} \II_{\bigcap_{i=1}^k\{\forall s\geq K: \log(Y_s^i)+s/2\in [-s^{2/3},s^{2/3}]\}} \nonumber \\ &\hspace{5cm} \times \exp\left( \int_0^\infty g((X^i_s,Y^i_s);1\leq i\leq k)ds\right)\bigg].
    \end{align}
    For $i\geq 2$, let $s_i$ be the last time that the mark $i$ is carried by the same particle as a mark $j$ with $j<i$. Set $s_1=0$.
    {Let} $\mathcal{G}=\sigma\left(\left\{s_i,Y_{s_i}^i,s\geq0,1\leq i \leq k\right\}\right)$. Conditional on $\mathcal{G}$, we have 
    \begin{align*}
        \PP\left(\forall i\leq k: X_\infty^i \in I \middle\vert \mathcal{G}\right) \leq \prod_{i=1}^k \PP_{(0,Y_{s_i}^i)}\left(X_\infty^i\in I\right) \leq C \prod_{i=1}^k \left(\left(\frac{\eps}{Y_{s_i}^i}\right)\wedge 1\right),
    \end{align*}
    where we used Lemma \ref{lem:bound_on_bm}. Assume that for all $i\geq 2$ we have that $s_i\geq K$. Then on this event we have control on $Y_{s_i}^i$, therefore
    \begin{equation*}
        \PP\left(\forall i\leq k: X_\infty^i \in I \middle\vert \mathcal{G}\right) \leq C \eps \prod_{i=2}^k \left(\left(\eps e^{s_i/2+s_i^{2/3}}\right)\wedge 1\right).
    \end{equation*}
    In fact, this still holds if $s_i\leq K$ by changing $C$.
    Using the explicit representation of $g$ this becomes
    \begin{align*}
        \EE\left[\mu^K(I)^k\right] &\leq C \eps \int_{\RR_+^{k-1}} \left[\prod_{i=2}^k \left(\left(\eps e^{s_i/2+s_i^{2/3}}\right)\wedge 1 \right)e^{-\beta s_i}\right] ds_2\ldots ds_{k}\\
        &= C\eps \left(\int_0^\infty \big(\eps e^{s/2+s^{2/3}} \wedge 1\big)e^{-\beta s}ds\right)^{k-1} \\
        &\leq C \eps\left(\eps+ \eps^{2\beta -\delta/(k-1)} \right)^{k-1},
    \end{align*}
    for any $\delta >0$ and some $C$. 
\end{proof}

From Proposition \ref{prop:2nd_moment} we derive the following corollary which is Theorem \ref{thm:main:2} (i) and (ii). We stated Theorem \ref{thm:main:2} (ii) for $\mu_\infty$, and consequently also $F$, defined on $\partial \DD$. The following corollary is stated for $\mu_\infty$ viewed on $\partial\HH$. {T}his is equivalent because the isometry that maps $\HH$ to $\DD$ is a diffeomorphism and thus preserves Hölder--continuity.

\begin{cor}\label{cor:hoelder} \hspace{1cm}
    \begin{enumerate}
    \item 
    Consider $F(x)=\mu_\infty((-\infty,x])$, the (random) cumulative distribution function of $\mu_\infty$. Then $F$ is almost surely Hölder--continuous for every exponent $\gamma < (1/2) \wedge (\beta/6)$. In particular, $\mu_\infty$ has no atoms almost surely. 

    \item 
    Consider $F^K(x)=\mu^K_\infty((-\infty,x])$, the (random) cumulative distribution function of $\mu^K_\infty$. Then $F^K$ is almost surely Hölder--continuous for every exponent $\gamma < 1 \wedge 2\beta$.
    \end{enumerate}
\end{cor}

\begin{proof}
    The key idea of this proof is to look at $(x\mapsto F(x))$ as {a} stochastic process to which we can apply Kolmogorov's continuity theorem. By Proposition \ref{prop:2nd_moment} we have {that} for $\eps$ small enough, uniformly in $x$,
    \begin{equation*}
        \EE\left[\vert F(x+\eps)-F(x)\vert^2 \right] \leq C\eps^{2\wedge (1+\beta/3)}.
    \end{equation*}
    For $\beta=3$ we use $\eps^{2-\delta}$ for arbitrarily small $\delta>0$.
    It now follows immediately from Kolmogorov's continuity theorem (see for example \cite[Theorem 4.23]{Kal2021}) that $F$ is almost surely continuous because $F$ is non-decreasing and has càdlàg paths. Further, $F$ is Hölder continuous for any $\gamma< (1/2)\wedge (\beta/6)$.

    We can improve on the bound on $\gamma$ if we consider $F^K$. The same reasoning applies but if we use Lemma \ref{lemma:k_th_moment} instead of Proposition \ref{prop:2nd_moment} then Kolmogorov's continuity theorem provides us with Hölder continuity for any $\gamma$ with 
    \begin{equation*}
        \gamma < \frac{k-1}{k} \wedge \frac{2\beta(k-1)-\delta}{k}.
    \end{equation*}
    Because $k\geq 2$ and $\delta>0$ are arbitrary, we have Hölder continuity for any $\gamma < 1\wedge 2\beta$.
\end{proof}

We also complete the proof of Theorem \ref{thm:main:1} by showing a lower bound on the Hausdorff dimension.

\begin{cor}
    For any $\beta>0$, we almost surely have that $\dim \mu_\infty \geq 2\beta \wedge 1$.
\end{cor}

\begin{proof}
    By Proposition \ref{prop:conv_typical_to_mu}, we have for any $K$ that $\mu^K_\infty \leq \mu_\infty$ and hence $\dim \mu_\infty \geq \dim  \mu_\infty^K.$
    Choose $K$ large enough so that $\mu_\infty^K$ is a non-trivial measure. This is almost surely possible. Let $F^K$ be the cumulative distribution function for $\mu_\infty^K$ as in Corollary \ref{cor:hoelder}. It is a basic fact of Hausdorff dimension that $F^K$ being Hölder continuous with exponent $\gamma$ implies that 
    \begin{equation*}
        \dim  \mu_\infty^K \geq \gamma.
    \end{equation*}
    This is essentially a consequence of Frostman's Lemma, Lemma \ref{lemma:frostman}, {which} is called the mass distribution principle \cite[Principle 4.2]{falconer_fractal_2014}. This completes the proof as we can choose any $\gamma < 2\beta \wedge 1$ by Corollary \ref{cor:hoelder}.
\end{proof}

Lastly, we prove Theorem \ref{thm:main:2} (iii).

\begin{cor}
    If $\beta > 1/2$, then $\mu_\infty$ almost surely has a density with respect to the Lebesgue measure.
\end{cor}

\begin{proof}
    We first show that for any $K$, $\mu_\infty^K$ has a density with respect to the Lebesgue measure.
    In the regime $\beta> 1/2$, we can choose $\delta$ small enough so that we have $1+2\beta -\delta\geq 2$, hence Proposition \ref{prop:2nd_moment} states 
    \begin{equation}\label{eq:bb}
        \sup_{\substack{\vert I \vert =\eps}} \EE\left[ \mu^K_\infty(I)^2\right] \leq C \eps^2,
    \end{equation}
    uniformly in $\eps$ small. 
    
    We construct a density of $\mu^K_\infty$ by approximations. Let $R,n\in \NN$. Define
    \begin{equation*}
        \rho_n^{K,R}(x) = \sum_{k=-nR}^{nR-1} \frac{\mu^K_\infty([k/n, (k+1)/n))}{
        1/n
        }1_{\{x \in [k/n,(k+1)/n)\}}.
    \end{equation*}
    We think of $\rho_n^{K,R}$ as an approximation to the density of $\mu_\infty^K$ on $[-R,R]$. We compute the expected $L^2$ norm of $\rho_n^{K,R}$,
    \begin{align*}
        \EE\left[\|\rho_n^{K,R}\|_2^2\right] = \sum_{k=-nR}^{nR-1} \frac{1}{n} \EE\left[\frac{\mu^K_\infty([k/n, (k+1)/n))^2}{ 1/n^2}\right] \leq \sum_{k=-nR}^{nR-1} \frac{C}{n} =2RC,
    \end{align*}
    where we used \eqref{eq:bb}. Note that this bound is uniform in $n$. This also means that for every $L>0$, by Markov's inequality, 
    \begin{equation}\label{eq:cpt}
        \PP\left(\|\rho_n^{K,R}\|_2 >L \right) \leq \frac{(2RC)^2}{L^2} \xrightarrow{L \to 0} 0.
    \end{equation}
    By the Banach--Alaoglu theorem, sets of the form $\{f\in L^2([-R,R]): \|{f}\|_2 \leq L\}$ are compact in the weak topology. This means by \eqref{eq:cpt} that the sequence $(\rho_n^{K,R})_n$ is tight in $L^2([-R,R])$
    and by Prokhorov's theorem there exists a weakly convergent subsequence, call its limit $\rho^{K,R}$.

    On the other hand, let $\mu^{R,n}_\infty$ be the measure induced by the density $\rho_n^{K,R}$. Clearly, $\mu^{K,R,n}_\infty$ converges almost surely weakly to $\mu^K_\infty \vert_{[-R,R]}$. Hence $\rho^{K,R}$ is a density for $\mu^K\vert_{[-R,R]}$. As $R\to \infty$, $\mu^K_\infty\vert_{[-R,R]}$ converges weakly to $\mu^K_\infty$. By a diagonal argument one can see that $\rho^{K,R}$ converges weakly to a function $\rho^K$ which is a density for $\mu^K$. 

    We now turn to $\mu_\infty$. For $K<K'$ we have that $\rho^K \leq \rho^{K'}$ almost everywhere because $\mu^K_\infty \leq \mu^{K'}_\infty$ almost surely by Proposition \ref{prop:conv_typical_to_mu}. Define now $\rho = \lim_{K \to \infty} \rho^K$ taken along the sequence $K\in \NN$, by monotonicity this limit exists almost everywhere and $\rho<\infty$ almost everywhere. Because $\mu^K_\infty$ converges weakly to $\mu_\infty$ as $K\to \infty$ by Proposition \ref{prop:conv_typical_to_mu}, we get that $\rho$ is a density for $\mu_\infty$. 
\end{proof}

\section{Open questions}\label{sec:open_questions}

\begin{quest}
    For $\beta=1/2$, show that $\mu_\infty$ does not admit a density with respect to the Lebesgue measure. This is because we believe Proposition \ref{prop:2nd_moment} to be sharp, that is we cannot achieve the exponent $\eps^2$. 
\end{quest}

\begin{quest}
    In Proposition \ref{prop:conv_typical_to_mu} we have shown that $\mu_\infty$ is determined by particles that satisfy $Y_u(t)\approx e^{-t/2}$. It would be interesting look at the empirical measure of particles that satisfy $Y_u(t)\approx e^{-t/2+\lambda t}$. This should be non--trivial for any $\lambda$ with $\vert \lambda \vert < \sqrt{2\beta}$. In particular, if you look at $\lambda=1/2$ there should {be} a phase transition when $\beta=1/8$ because this is the threshold for local survival. For BBM on $\RR$, the  number of particles {at a given speed} is counted by the additive martingale; note that here in the hyperbolic setting
    \begin{equation*}
        \sum_{u\in \mathcal{N}(t)} (Y_u(t))^\lambda e^{-t(\lambda^2+\lambda-2)/2}
    \end{equation*}
    has the same distribution as the regular additive martingale. 
\end{quest}

\begin{quest}
    It would be interesting to study $\dim \Lambda$ and $\dim \mu_\infty$ in the presence of a drift: assume that there is a drift of strength $\lambda \in \RR$ away from the origin. If $\lambda>-1/2$, $\mu_\infty$ {should} still be supported on the boundary. For $\lambda <-1/2$, we {should} almost surely have that $\mu_\infty$ is a Dirac mass at the origin, and for $\lambda=-1/2$, $\mu_t$ {should} not converge. 

    If we were to consider drift away from a point $\zeta\in \partial\DD$, the analysis becomes easy. By isometry, we can choose $\zeta =1${,} which corresponds to the unique boundary point $\infty$ at infinity in $\partial \HH$. Geodesics going through $\infty$ in $\HH$ are straight vertical lines. This means drift away from $\infty$ is a simple vertical drift of $\lambda$ (weighted by the hyperbolic metric). This means that {the} calculations in this work and in \cite{lalley_hyperbolic_1997} still apply and we should get for $\lambda>-1/2$,
    \begin{equation*}
    \dim  \mu_\infty(\lambda) = [2\beta/(2\lambda+1)]\wedge 1,
    \end{equation*}
    and
    \begin{equation*}
        \dim \Lambda(\lambda) =
        \begin{cases}
            \frac{1}{2}\left(1+2\lambda - \sqrt{(1+2\lambda)^2-8\beta}\right) \ &\text{for }\beta \leq \frac{(1+2\lambda)^2}{8}, \\
            1 & \text{else}.
        \end{cases}
    \end{equation*}
    We believe the same expressions should still hold for drift away from the origin rather than from the boundary. The reason for this is that for $z\in \DD$ we can replace drift away from the origin with drift away from $-\frac{z}{\vert z\vert}$. If we start a hyperbolic Brownian motion $Z_t$ at $z$ where $z$ is far away from the origin, then $-\frac{Z_t}{\vert Z_t\vert}$ should not vary much, hence we can replace drift away from the origin with drift away from a boundary point. 
\end{quest}

\begin{quest}
    In this paper, the underlying stochastic process from which we build the branching process is continuous in time and space. It is also natural to consider a discrete setting, that is a branching random walk on a hyperbolic group. Let $\Gamma$ be a non--elementary hyperbolic group, generated by a finite set $S$. Given a probability measure $\nu$ on $S$ with $\supp \hspace{-3pt}(\nu) =S$, we can then construct a random walk on $\Gamma$ by setting $p(x,y)=\nu(x^{-1}y)$, and hence a branching random walk (BRW) with {branching} rate $\beta>0$. Let $\vert x \vert$ be the norm of $x\in \Gamma$ in the word metric induced by $S$. Then for the random walk induced by $\nu$, $(X_n,n\geq 0)$, there {are} $\sigma^2,v>0$ such that
    \begin{equation*}
            \frac{\vert X_n\vert}{n}\xrightarrow[n\to\infty]{a.s.}v \qquad \text{and} \qquad
            \frac{\vert X_n\vert - nv}{\sqrt{\sigma^2n}} \xrightarrow[n\to \infty]{d} \mathcal{N}(0,1),
    \end{equation*}
    see for example \cite{bjorklund_central_2010} and the references therein.
    That is, $(\vert X_n\vert)_{n\geq 0}$ satisfies a strong law of large numbers and a central limit theorem. Like the hyperbolic plane $\HH$, $\Gamma$ can be endowed with a natural boundary $\partial \Gamma$ with metric $d_a$ where $a>1$ is the visual parameter. In this setting, we can again consider $\Lambda$, the set of accumulation points of the BRW on $\partial \Gamma$, and $\mu_\infty$, the limit of the empirical measure. There have been multiple works studying $\Lambda$ and its dimension, for example \cite{sidoravicius_limit_2023}, but none studying $\mu_\infty$. We believe that the methods from this paper {should} transfer easily to this setting, in particular because the LLN and CLT guarantee that an equivalent of Proposition \ref{prop:conv_typical_to_mu} still holds true. We believe the following should be true. 

    \begin{conj}
        In this setting, $\dim  \mu_\infty = \left(\frac{\beta}{v\log(a)} \right)\wedge \dim \partial \Gamma$ almost surely. 
    \end{conj}

    Note that $\dim \partial \Gamma = \frac{\delta}{\log a}$ where $\delta$ is the exponential growth rate of the volume of $\Gamma$. This would be in contrast to the complicated expressions for $\dim \Lambda$, for example determined by \cite{sidoravicius_limit_2023} and \cite{hueter_anisotropic_2000}.
\end{quest}

\section*{Acknowledgements}

I would like to thank Julien Berestycki for making me aware of this problem. I would also like to thank Nathanaël Berestycki for making us aware of Woess' article. 
Further, I am grateful for the support of EPSRC grant EP/W523781/1.

\printbibliography[heading=bibintoc]

@article {Bra1978,
    AUTHOR = {Bramson, Maury D.},
     TITLE = {Maximal displacement of branching {B}rownian motion},
   JOURNAL = {Communications on Pure and Applied Mathematics},
  FJOURNAL = {Communications on Pure and Applied Mathematics},
    VOLUME = {31},
      YEAR = {1978},
    NUMBER = {5},
     PAGES = {531--581},
      ISSN = {0010-3640},
   MRCLASS = {60J80},
  MRNUMBER = {494541},
MRREVIEWER = {S\o ren Asmussen},
       DOI = {10.1002/cpa.3160310502},
       URL = {https://doi.org/10.1002/cpa.3160310502},
}

@article {AidBerBruShi2013,
    AUTHOR = {Aïd\'{e}kon, E. and Berestycki, J. and Brunet, \'{E}. and Shi, Z.},
     TITLE = {Branching {B}rownian motion seen from its tip},
   JOURNAL = {Probab. Theory Related Fields},
  FJOURNAL = {Probability Theory and Related Fields},
    VOLUME = {157},
      YEAR = {2013},
    NUMBER = {1-2},
     PAGES = {405--451},
      ISSN = {0178-8051},
   MRCLASS = {60J80 (60G55 60G70)},
  MRNUMBER = {3101852},
MRREVIEWER = {J\'{a}nos Engl\"{a}nder},
       DOI = {10.1007/s00440-012-0461-0},
       URL = {https://doi.org/10.1007/s00440-012-0461-0},
}

@article{candellero_boundary_2023,
	title = {On the boundary at infinity for branching random walk},
	volume = {28},
	journal = {Electronic Communications in Probability},
	author = {Candellero, Elisabetta and Hutchcroft, Tom},
	year = {2023},
	pages = {1--12},
}

@article{dussaule_branching_2025,
	title = {Branching random walks on relatively hyperbolic groups},
	volume = {53},
	issn = {0091-1798},
	url = {https://projecteuclid.org/journals/annals-of-probability/volume-53/issue-2/Branching-random-walks-on-relatively-hyperbolic-groups/10.1214/24-AOP1708.full},
	doi = {10.1214/24-AOP1708},
	number = {2},
	urldate = {2025-07-31},
	journal = {The Annals of Probability},
	author = {Dussaule, Matthieu and Wang, Longmin and Yang, Wenyuan},
	month = mar,
	year = {2025},
}

@book{falconer_fractal_2014,
	address = {Newark},
	edition = {3rd ed},
	series = {New {York} {Academy} of {Sciences} {Series}},
	title = {Fractal {Geometry}: {Mathematical} {Foundations} and {Applications}},
	shorttitle = {Fractal {Geometry}},
	publisher = {John Wiley \& Sons, Incorporated},
	author = {Falconer, Kenneth J.},
	year = {2014},
}

@article{harris_many--few_2017,
	title = {The many-to-few lemma and multiple spines},
	volume = {53},
	issn = {0246-0203},
	url = {https://projecteuclid.org/journals/annales-de-linstitut-henri-poincare-probabilites-et-statistiques/volume-53/issue-1/The-many-to-few-lemma-and-multiple-spines/10.1214/15-AIHP714.full},
	doi = {10.1214/15-AIHP714},
	number = {1},
	urldate = {2024-01-31},
	journal = {Annales de l'Institut Henri Poincaré, Probabilités et Statistiques},
	author = {Harris, Simon C. and Roberts, Matthew I.},
	month = feb,
	year = {2017},
}

@article{bjorklund_central_2010,
	title = {Central {Limit} {Theorems} for {Gromov} {Hyperbolic} {Groups}},
	volume = {23},
	number = {3},
	journal = {Journal of Theoretical Probability},
	author = {Björklund, Michael},
	year = {2010},
	pages = {871--887},
}

@book {Kal2021,
    AUTHOR = {Kallenberg, Olav},
     TITLE = {Foundations of modern probability},
    SERIES = {Probability Theory and Stochastic Modelling},
    VOLUME = {99},
        edition = {third},
 PUBLISHER = {Springer, Cham},
      YEAR = {2021},
     PAGES = {xii+946},
      ISBN = {978-3-030-61871-1; 978-3-030-61870-4},
   MRCLASS = {60-01 (60A10 60G05)},
  MRNUMBER = {4226142},
MRREVIEWER = {Myron Hlynka},
       DOI = {10.1007/978-3-030-61871-1},
       URL = {https://doi.org/10.1007/978-3-030-61871-1},
}

@article{hueter_anisotropic_2000,
	title = {Anisotropic branching random walks on homogeneous trees},
	volume = {116},
	number = {1},
	urldate = {2025-07-31},
	journal = {Probability Theory and Related Fields},
	author = {Hueter, Irene and Lalley, Steven P.},
	year = {2000},
	pages = {57--88},
}

@article{kaimanovich_limit_2023,
	title = {Limit distributions of branching {Markov} chains},
	volume = {59},
	number = {4},
	urldate = {2024-07-17},
	journal = {Annales de l'Institut Henri Poincaré, Probabilités et Statistiques},
	author = {Kaimanovich, Vadim A. and Woess, Wolfgang},
	year = {2023},
	pages = {1951--1983},
}

@article{lalley_hyperbolic_1997,
	title = {Hyperbolic branching {Brownian} motion},
	volume = {108},
	number = {2},
	journal = {Probability Theory and Related Fields},
	author = {Lalley, Steven P. and Sellke, Tom},
	year = {1997},
	pages = {171--192},
}

@article{sidoravicius_limit_2023,
	title = {Limit {Set} of {Branching} {Random} {Walks} on {Hyperbolic} {Groups}},
	volume = {76},
	number = {10},
	journal = {Communications on Pure and Applied Mathematics},
	author = {Sidoravicius, Vladas and Wang, Longmin and Xiang, Kainan},
	year = {2023},
	pages = {2765--2803},
}

@article{woess_notes_2024,
	title = {Notes on hyperbolic branching {Brownian} motion},
	publisher = {arXiv},
        journal = {Preprint},
	author = {Woess, Wolfgang},
	year = {2024},
	note = {arXiv:2405.07301},
}

\end{document}